\documentclass[11pt,psamsfonts,reqno]{amsart}
\usepackage{amsmath}
\usepackage{amsthm}
\usepackage{amssymb}
\usepackage{amscd}
\usepackage{amsfonts}
\usepackage{amsbsy}
\usepackage{epsfig,afterpage}
\usepackage[hidelinks]{hyperref}
\usepackage[dvips]{psfrag}
\usepackage{xcolor}     
\usepackage{overpic}

\usepackage{float}
\usepackage{tikz, wrapfig}
\usetikzlibrary{matrix,arrows,arrows.meta,bending,decorations.pathreplacing,decorations.pathmorphing,decorations.markings,fit,patterns,shapes,intersections,calc}
\tikzset{node distance=3cm, auto}
\allowdisplaybreaks

\newcommand{\R}{\ensuremath{\mathbb{R}}}

\newcommand{\N}{\ensuremath{\mathbb{N}}}

\newcommand{\Z}{\ensuremath{\mathbb{Z}}}

\newcommand*{\ponto}{\makebox[1.5ex]{\textbf{$\cdot$}}}%

\newtheorem {theorem} {Theorem} [section]
\newtheorem {proposition}  {Proposition} [section]

\newtheorem {lemma}  {Lemma}[section]
\newtheorem {definition}  {Definition}[section]

\newtheorem {example} {Example}[section]

\newtheorem {corollary}  {Corollary}[section]

\begin{document}



\title{On Topological Entropy of Piecewise Smooth Vector Fields}
\author[A. A. Antunes, T.  Carvalho and R. Var\~ao.]{André  Amaral Antunes$^1$, Tiago Carvalho$^2$  and Régis Varão$^3$}

\address{$^1$ IBILCE/UNESP, S\~ao Jose do Rio Preto, S\~ao Paulo, Brazil} \address{$^2$ FFCLRP-USP, Ribeir\~ao Preto, S\~ao Paulo, Brazil}

\address{$^3$ IMECC-UNICAMP, Campinas, S\~ao Paulo, Brazil}
\email{andre.antunes@unesp.br}
\email{tiagocarvalho@usp.br}
\email{varao@unicamp.br}



\begin{abstract}
Non-smooth vector fields does not have necessarily the property of uniqueness 
of solution passing through a point and this is responsible to enrich the 
behavior of the system. Even on the plane non-smooth vector fields can 
be chaotic, a feature impossible for the smooth or continuous case. We 
propose a new approach towards a better understanding of chaos for non-smooth 
vector fields and this is done by studying the entropy of the system.

In this work we set the ground for one to begin the study of entropy for 
non-smooth vector fields. We construct a metric space of all possible 
trajectories of a non-smooth vector field, where we define a flow inherited by 
the vector field and then define the topological entropy in this scenario. As a 
consequence, we are able to obtain some general results of this theory and give 
some examples of planar non-smooth vector fields with positive (finite and 
infinite) entropy. 
\end{abstract}

\maketitle

\section{Introduction}\label{intro}


A large range of real problems can be modelled using systems of Ordinary 
Differential Equations (ODEs for short) (or, equivalently, vector fields). Some 
of these problems are governed by two, or more, systems; for example: in a 
cancer treatment, there exist a system of ODEs that model the dynamic when the 
patient is submmited to a severe drug treatment (and the infected cells decrease 
according to a law of evolution in this process) and there is another system of 
ODEs governing the dynamic when this treatment is abruptly interrupted in order 
to provide the  immune system recovery 
(\cite{CarvalhoGoncalvesManceraRodrigues-1-2019-cancer, 
CarvalhoGoncalvesManceraRodrigues-2-2020-cancer}). The same holds is recent 
protocols of HIV treatment 
(\cite{CarvalhoGoncalvesCristianoTonon-NonLiDy-2020-HIV}). There is also 
examples in ecology, modelling predator preys populations when predator suddenly 
change its food preference (\cite{CarvalhoGoncalvesNovaes-2020-Shilnikov}); and 
electric engineering(\cite{Cristiano2019, Wang2020}).

These vector fields governed by two or more laws are known as \textit{Non-smooth 
vector fields} (NSVFs) or \textit{Piecewise smooth vector fields} (PSVF). The 
general mathematical theory about them is being constructed in recent years. 
There are some similarities with the smooth vector fields (e.g. the Poincar\'e- 
Bendixson Theorem \cite{BCEminimal}), but there are many behaviours which are 
typical to NSVFs that are not possible in the smooth case. Describing this has 
been a main source of research for the mathematicians dedicated to this area. 
For example: it is possible to construct a NSVF on the plane that is chaotic 
(\cite{BCEchaotic}), whereas this feature is impossible for the smooth case. The 
main aspect that leads to these differences is the non uniqueness of 
trajectories passing through a point of the domain of the NSVF.

Although NSVF have already a long history it is a field with still many natural 
open questions. Even basic questions concerning topological transitivity have 
not been fully answered. For instance it is known that continuous vector field 
cannot be transitive on the $S^2$ which recently has been shown not to be 
true in the case for NSVF 
\cite{EJV}. In a recent work \cite{EV} the authors establish some basic 
questions concerning topological transitivity and the existence of a trajectory 
which is dense, also that topological transitivity is equivalent to chaotic 
behavior even in some general two dimensional case. These results illustrate 
that NSVF has a long way ahead to be better understood. What we do in this work 
is to apply the basic concepts Ergodic Theory to the NSVF. A first question 
question one must solve is the invariant measures, this has already been done 
(see \cite{NV} and references therein). Entropy plays a major role in Ergodic 
Theory as it is a very successful tool, therefore and that is what we are 
going to explore here.

In Section \ref{sec:NSVFs}, we provide the main ideas and general definitions 
concerning NSVFs that we use along this text; for more details, 
\cite{diBernardo-livro} is a good textbook for general theory and some 
applications, and also \cite{guardia} and \cite{kuznetsov} are very 
comprehensive on theory of NSVF.

In this work, we construct a metric space of all possible trajectories of a NSVF, where we can define a flow inherited by the non-smooth vector field and use it to define the topological entropy of the NSVF.  The difficult in doing it is that the \textit{classical} definition of entropy for a flow relies on uniqueness of trajectory to construct an appropriate transformation of the phase space and then using it to define the entropy. To the best of our knowledge,  there is no work related to the topological entropy of NSVFs.
As consequence, we are able to obtain planar NSVFs with positive (and even, infinite) entropy. In some sense, this explain the complicated dynamic obtained in NSVFs, even on the plane.

\section{Basic Notions on NSVFs}\label{sec:NSVFs}

Let $ U\subset\R^{2} $ an open set, and consider a codimension one manifold $\Sigma$ of $\R^2$ given by $\Sigma =f^{-1}(0)\cap U,$ where $f:\R^2 \rightarrow \R$ is a $ \mathcal{C}^{r} $ function (with $ r>1 $ large enough for our purposes) having $ 0\in \R $ as a regular value. We call $\Sigma$ the \textit{switching manifold} and it splits $ U $ in two open sets: $\Sigma^+=\{q\in U \, | \, f(q) > 0\}$ and $\Sigma^-=\{q \in U\,\,|\,\, f(q)< 0\}$. A NSVF defined in $ \R^{2} $ is a map of the form:
\begin{equation}\label{eq Z}
Z(q)=\left\{
\begin{array}{l}
X(q),\quad $for$ \quad q \in \Sigma^+,\\ 
Y(q),\quad $for$ \quad q \in \Sigma^-.
\end{array}\right.
\end{equation}
where $ X $ and $ Y $ are $ \mathcal{C}^{r} $ vector fields in $ \overline{\Sigma^{+}} $ and $ \overline{\Sigma^{-}} $. In this, we are using the standard convention that for a function to be smooth in a non-open domain $ D $, it means that it can be extended to a smooth function on an open set containing $ D $.

We denote (\ref{eq Z}) simply as $ Z=(X,Y) $. And let $ \mathcal{Z}^{r} $ denote the space of all vector fields of this type. We consider this space with the product topology between topologies of $ \mathfrak{X}^{r}(\overline{\Sigma^{+}}) $ and $ \mathfrak{X}^{r}(\overline{\Sigma^{-}}) $.

In order to define rigorously the flow of $ Z $ we distinguish whether this point is at $ \Sigma^{\pm} $ or $ \Sigma $. For the first two regions, the local trajectory is defined by $ X $ and $ Y $ respectively, as usual, but for $ \Sigma $ we divide this into three regions based on the contact between the vector fields $ X,Y $ and $ \Sigma $ characterized by the Lie derivative $ Xf(p)=\left\langle \nabla f(p), X(p)\right\rangle $  where $\langle \ponto , \ponto \rangle$ is the usual inner product:

$\bullet$ Crossing Region: $\Sigma^c=\{ p \in \Sigma \, | \, Xf(p)\cdot Yf(p)> 0 \}$. Moreover, we denote $\Sigma^{c+}= \{ p \in \Sigma \,\,|\,\, Xf(p)>0, Yf(p)>0 \}$ and $\Sigma^{c-} = \{ p \in \Sigma \,\,|\,\, Xf(p)<0,Yf(p)<0 \}$.

$\bullet$ Sliding Region: $\Sigma^{s}= \{ p \in \Sigma \,\,|\,\, Xf(p)<0, Yf(p)>0 \}$.

$\bullet$ Escaping Region: $\Sigma^{e}= \{ p \in \Sigma \,\,|\,\, Xf(p)>0 ,Yf(p)<0\}$.

These regions are relatively open in $ \Sigma $ and their definitions exclude the tangency points, where $ Xf(p)\cdot Yf(p)=0 $. These points are on the boundary of those regions.
 
\begin{definition}\label{def folds}
	A tangency point $p \in \Sigma$ is a \emph{fold point} of $ X $ if $ Xf(p)=0 $ but $X^{2}f(p)\neq0$, where  $ X^{i}f(p)=\left\langle \nabla X^{i-1}f(p), X(p)\right\rangle $. Moreover, $ p\in\Sigma $ is a \emph{visible} (respectively \emph{invisible}) fold point of $X$ if $Xf(p)=0$ and $X^{2}f(p)> 0$ (respectively $X^{2}f(p)< 0$). A point $ p\in\Sigma $ is a two-fold, if it is a fold point for both $ X $ and $ Y $, and it is visible-visible if visible for both (respectively invisible-visible and invisible-invisible).

\end{definition}

In addition, a tangency point $ p $ is \emph{singular} if $ p $ is a invisible tangency for both $ X $ and $ Y $.  On the other hand, a tangency point $ p $ is \emph{regular} if it is not singular.

In $ \Sigma^{s,e} $, the definition of local orbit is given by Filippov convention \cite{Fi}, described in the following.

\begin{definition}\label{definicao campo deslizante tangencial} 
	Given a point $ p\in\Sigma^s\cup\Sigma^e\subset\Sigma $, we define the \emph{sliding vector field}  at $ p $ as the convex combination of $ X(p) $ and $ Y(p) $ that is tangent to $ \Sigma $. In the planar case, it is given by the expression
	\begin{equation}\label{expfilipov}
	Z^{s}(p)=\frac{Yf(p)X(p)-Xf(p)Y(p)}{Yf(p)-Xf(p)}.
	\end{equation}
	
\end{definition}

Moreover, the sliding vector field can be extend to $\overline{\Sigma^e}\cup\overline{\Sigma^s}$.

Now we establish the classical convention on the trajectories of a PSVF:

\begin{definition}\label{definicao trajetorias}
	The \emph{local trajectory} $ \phi_{Z}(t,p) $ of a PSVF
	$ Z=(X,Y) $ through $p\in U$ is defined as follows:
	\begin{itemize}
		\item[(i)] For $ p\in\Sigma^+ $ and $ p\in\Sigma^{-} $ the trajectory is given by $ \phi_{Z}(t,p)=\phi_{X}(t,p) $ and
		$ \phi_{Z}(t,p)=\phi_{Y}(t,p) $ respectively, where $ t\in I $.
		
		\item[(ii)] For $ p\in\Sigma^{c+} $ and  taking the origin of time at $ p $, the trajectory is defined as $ \phi_{Z}(t,p)=\phi_{Y}(t,p) $ for $ t\in I\cap\{t\leq 0\} $ and $ \phi_{Z}(t,p)=\phi_{X}(t,p) $ for $ t\in I\cap\{t\geq 0\} $. If $ p\in\Sigma^{c-} $  the definition is the same
		reversing time.

		\item[(iii)] For $ p\in\Sigma^e $ and taking the origin of time at $ p $, the trajectory is defined as $ \phi_{Z}(t,p)=\phi_{Z^{s}}(t,p)$ for $t\in I\cap \{t\leq 0\} $ and $ \phi_{Z}(t,p) $ is either $ \phi_{X}(t,p) $ or $ \phi_{Y}(t,p) $ or $ \phi_{Z^{s}}(t,p) $ for $ t\in I\cap \{t\geq 0\} $. For $ p\in\Sigma^s $ the definition is the same reversing time.

		\item[(iv)] For $p$ a regular tangency point and  taking the
		origin of time at $p$, the trajectory is defined as
		$\phi_{Z}(t,p)=\phi_{1}(t,p)$ for $t\in I\cap \{t\leq 0\}$ and
		$\phi_{Z}(t,p)=\phi_{2}(t,p)$ for $t\in I\cap \{t\geq 0\}$, where each $\phi_{1},\phi_{2}$ is either $\phi_{X}$ or $\phi_{Y}$ or $\phi_{Z^{s}}$.
		
		\item[(v)]For $p$ a singular tangency point, $\phi_{Z}(t,p)=p$ for all $t \in \R$.
	\end{itemize}
\end{definition}

\begin{definition}\label{definicao trajetoria global}
	The \textbf{ global trajectory} (\textbf{orbit}) $\Gamma_Z (t,p_0)$ of $Z$ passing through $p_0$ is a union $$\Gamma_Z (t,p_0) = \bigcup_{i\in\mathbb{Z}}\{ \sigma_{i}(t,p_i) : t_i \leq t\leq t_{i+1} \}$$ of preserving-orientation local trajectories $\sigma_{i}(t,p_i)$ satisfying $\sigma_{i}(t_{i+1},p_i)=\sigma_{i+1}(t_{i+1},p_{i+1})=p_{i+1}$ and $t_i\rightarrow\pm\infty$ as $i\rightarrow\pm\infty$.
\end{definition}

\begin{definition}\label{defi-z-invariante}
	A set $ A $ is $ Z- $\emph{invariant} if for each $ p\in A $ and all maximal trajectory $ \Gamma_{Z}(t,p) $ passing through $ p $ it holds $ \Gamma_{Z}(t,p)\subset A $
\end{definition}

\section{Basic Notions on Entropy}\label{sec:entropy}

Let $ f:X\to X $ a continuous map of a compact metric space $ (X,d) $ and $f^i = f \circ f \circ \hdots \circ f$ ($i$-times). Define a sequence of metrics $ d^{f}_{n} $, $ n=1,2,\dots $, by:
\[ d^{f}_{n}(x,y)=\max\limits_{0\leqslant i \leqslant n-1}d(f^{i}(x),f^{i}(y)) . \]

Denote $ B_{f}(x,\varepsilon,n)=\{y\in X \,\,|\,\, d_{n}^{f}(x,y)<\varepsilon\} $ the open ball centered at $ x $, with radius $ \varepsilon $ with respect to metric $ d^{f}_{n} $.
\begin{definition}
	A set $ E\subset X $ is \emph{$ \varepsilon- $dense with respect to $ d^{f}_{n} $} (or $ (n,\varepsilon)-$dense), if $ X\subset \bigcup_{x\in E}B_{f}(x,\varepsilon,n). $ Then the \emph{$ \varepsilon- $capacity} of $ d_{n}^{f} $, denoted $ S_{d}(f,\varepsilon,n) $, is the minimal cardinality of a $ (n,\varepsilon)- $dense set, equivalently, the cardinality of a minimal $ (n,\varepsilon)- $dense set. 
\end{definition}

\begin{definition}
	Consider the exponential growth rate
	\[ h_{d}(f,\varepsilon):=\limsup_{n\to\infty}\dfrac{1}{n}\log S_{d}(f,\varepsilon,n) \]
	and define the \emph{topological entropy} of $ f $, denoted $ h(f) $ by:
	\[ h(f):=\lim_{\varepsilon \to 0} h_{d}(f,\varepsilon). \]
\end{definition}

\begin{definition}\label{definition entropy smooth}
	Let $ X $ is a vector field, and $ \varphi $ is the flow defined by this field, we define the \emph{time-one map} of this field by $ X^{1}(x)=\varphi(x,1) $. And the \emph{topological entropy of the flow} is defined by $ h(\varphi)=h(X^{1}) $.
\end{definition}

\begin{definition}
	Let $ f:M\to M $ and $ g:N\to N $ be two maps. A map $ H:M\to N $ is called a \emph{topological semi-conjugacy from $ f $ to $ g $} provided that (i) $ H $ is continuous, (ii) $ H $ is onto, and (iii) $ H\circ f=g\circ H $. We may also refer to $ f $ being \emph{topologically semi-conjugate to $ g $ by $ H $}. 
	
	The map $ H $ is called \emph{topological conjugacy} if it is a semi-conjugacy and (iv) $ H $ is one to one and onto and has a continuous inverse (so $ H $ is a homeomorphism). We also say that $ f $ and $ g $ are \emph{topologically conjugate} by $ H $, or simply that $ f $ and $ g $ are \emph{conjugate}.
\end{definition}

The following Propositions give some useful properties of topological entropy. Their proofs can be found at \cite{Katok03}.

\begin{proposition}\label{invariant conjugacy}
	Let $ f:M\to M $ and $ g:N\to N $ be two maps, and $ H:M\to N $ a semi-conjugacy from $ f $ to $ g $. Then $ h(f)\geqslant h(g) $. Moreover, if $ H $ is a topological conjugacy, then $ h(f)=h(g) $.
\end{proposition}\qed

\begin{proposition}\label{prop-closed-f-invariant}
	If $ \Lambda $ is a closed $ f- $invariant set, then $ h(f\vert_{\Lambda})\leqslant h(f) $
\end{proposition}\qed

\begin{proposition}\label{prop:m-entropy}
	Let $ f:M\to M $ a map. Then $ h(f^{m})=|m| h(f) $.
\end{proposition}\qed

\section{ General Theory on Entropy for NSVFs}



%

In this section, we adapt Definition \ref{definition entropy smooth} for the non-smooth case and prove some preliminary results of this theory. As far as the authors know, this is the first work to cover these aspects of non-smooth vector fields.


Let $ Z=(X,Y) $ a NSVF defined over a compact 2-dimensional surface $ M $ and $ \Omega=\{\gamma:$ global trajectory of $ Z \} $. 

\begin{definition}
	Let $ \Omega $ be the set of all global trajectories, as before. Define $ \rho:\Omega\times\Omega\to\R $ by:
	\[ \rho(\gamma_{1},\gamma_{2})=\sum_{i\in\Z}\frac{1}{2^{|i|}}\int_{i}^{i+1}|\gamma_{1}(t)-\gamma_{2}(t)|dt, \]
	where $ |\gamma_{1}(t)-\gamma_{2}(t)| $ denotes the distance between the points $ \gamma_{1}(t) $ and $ \gamma_{2}(t) $.
\end{definition}

\begin{proposition}
	The space $ (\Omega,\rho) $ is a metric space.
\end{proposition}
\begin{proof}
	Let $ \gamma_{1},\gamma_{2}\in \Omega $. Observe that $ M $ being compact, implies $ |\gamma_{1}(t)-\gamma_{2}(t)| $ is bounded for all $ t\in\R $, thus the series above converges for any $ \gamma_{1},\gamma_{2} $.
	
	If $ \rho(\gamma_{1},\gamma_{2})=0 $ then $ \int_{i}^{i+1}|\gamma_{1}(t)-\gamma_{2}(t)|dt=0 $ for all $ i\in\Z $ which implies $ \gamma_{1}(t)=\gamma_{2}(t) $ for all $ t\in\R $ and therefore $ \gamma_{1}=\gamma_{2} $.
	
	The fact $ \rho(\gamma_{1},\gamma_{2})=\rho(\gamma_{2},\gamma_{1}) $ follows immediately from $ |\gamma_{1}(t)-\gamma_{2}(t)|=|\gamma_{2}(t)-\gamma_{1}(t)| $.
	
	And, finally, for the triangle inequality part it is enough to notice that $ |\gamma_{1}(t)-\gamma_{3}(t)|\leqslant |\gamma_{1}(t)-\gamma_{2}(t)|+|\gamma_{2}(t)-\gamma_{3}(t)| $ for all $ t\in\R $ gives the inequality $ \rho(\gamma_{1},\gamma_{2})\leqslant \rho(\gamma_{1},\gamma_{3})+\rho(\gamma_{3},\gamma_{2}) $.
\end{proof}

The Propositions \ref{lemma-epsilon-proximos}, \ref{lemma-converg-pontual} give some important insights about how this metric behaves.

\begin{proposition}\label{lemma-epsilon-proximos}
	Given $ t_{0}, \varepsilon_{0} $, then there exists $ \delta>0 $ such that, for all $ \gamma_{1},\gamma_{2} \in \Omega $ if $ \rho(\gamma_{1},\gamma_{2})<\delta $, then for every $ |t|<t_{0} $, $ |\gamma_{1}(t)-\gamma_{2}(t)|\leqslant \varepsilon_{0} $.
\end{proposition}
\begin{proof}
	Suppose, by contradiction, that for every $ \delta $, there is $-t_{0}< t_{\delta}<t_{0} $, such that $ |\gamma_{1}(t_{\delta})-\gamma_{2}(t_{\delta})|> \varepsilon_{0} $.
	
	By continuity, there is open interval $ J=(\alpha_{\delta},\beta_{\delta})\ni t_{\delta} $, such that for all $ t\in J $, $ |\gamma_{1}(t)-\gamma_{2}(t)|> \varepsilon_{0} $. Moreover, we know that $ \gamma_{1}, \gamma_{2} $ are differentiable by parts. Hence it makes sense to talk about their derivatives (which could differ if taken by left or right), and we know that because the vector fields that generate the Filippov system are all bounded, these derivatives are all uniformly bounded.
	
	Now,  
	\[ \rho(\gamma_{1},\gamma_{2})>\frac{1}{2^{t_{0}}}\int_{J}|\gamma_{1}(t)-\gamma_{2}(t)|dt>\frac{\varepsilon_{0}}{2^{t_{0}}}(\beta_{\delta}-\alpha_{\delta}) \Longrightarrow  \]
	
	\[\Longrightarrow \frac{\varepsilon_{0}}{2^{t_{0}}}(\beta_{\delta}-\alpha_{\delta})< \rho(\gamma_{1},\gamma_{2})<\delta. \]
	
	Which implies that $ \beta_{\delta}-\alpha_{\delta} $ goes to zero as $ \delta \to 0 $, and then the derivative of this function would explode at some point of $ J $, but it can not happen, since the derivatives are bounded.
	
\end{proof}

\begin{proposition}\label{lemma-converg-pontual}
	Given $ t_{0}>0 $, and $ \left(\gamma_{n}\right)\subset \Omega $ such that $ \gamma_{n}\to\gamma \in \Omega $, then for all $ |t|<t_{0} $, $ \gamma_{n}(t)\to\gamma(t) $.
\end{proposition}
\begin{proof}
	Let $ \varepsilon>0 $ given. There exists $ \delta>0 $, such that
	$$ \rho(\gamma_{n},\gamma)<\delta \Rightarrow |\gamma_{n}(t)-\gamma(t)|< \varepsilon, \,\, \forall |t|<t_{0}. $$
	
	And there is $ N\in\N $ such that, if $ n>N $ then $ \rho(\gamma_{n},\gamma)<\delta $. Hence, if $ n>N $, then $ |\gamma_{n}(t)-\gamma(t)|< \varepsilon $, for every $ |t|<t_{0} $.
	
\end{proof}

In this space of all orbits, we can define a natural flow determined by the NSVF $ Z $ as $ F:\R\times\Omega \to\Omega $, $ F(t,\gamma)(.)=\gamma(.+t) $. And then we have the \emph{time-one map}, $ F_{1} (\gamma)=\gamma(.+1) $.

\begin{proposition}
	The map $ F_{1}:\Omega\to\Omega $ defined above is a homeomorphism.
\end{proposition}
\begin{proof} 
	The function $ F_{1} $ is invertible, with inverse $ \left(F_{1}\right)^{-1}(\gamma)(.)=\gamma(.-1) $.
	
	Now,  
	\begin{align*}
	\int_{i}^{i+1}|F_{1}(\gamma_{1})(t)-F_{1}(\gamma_{2})(t)|dt&= \int_{i}^{i+1}|\gamma_{1}(t+1)-\gamma_{2}(t+1)|dt = \\
	&= \int_{i+1}^{i+2}|\gamma_{1}(t)-\gamma_{2}(t)|dt.
	\end{align*}
	
	Using the relation above, we obtain:
	\begin{align*}
	\rho(F_{1}(\gamma_1)&,F_{1}(\gamma_2))=\sum_{i\in\Z}\frac{1}{2^{|i|}}\int_{i}^{i+1}|F_{1}(\gamma_{1})(t)-F_{1}(\gamma_{2})(t)|dt=\\
	&=\sum_{i\in\Z}\frac{1}{2^{|i|}}\int_{i+1}^{i+2}|\gamma_{1}(t)-\gamma_{2}(t)|dt=\\
	&= \lim\limits_{k\to\infty}\left(\sum_{i=0}^{k}\frac{1}{2^{i}}\int_{i+1}^{i+2}|\gamma_{1}(t)-\gamma_{2}(t)|dt+ \sum_{i=-k}^{-1}\frac{1}{2^{-i}}\int_{i+1}^{i+2}|\gamma_{1}(t)-\gamma_{2}(t)|dt \right)=\\
	&= \lim\limits_{k\to\infty}\left(2\sum_{i=0}^{k}\frac{1}{2^{i+1}}\int_{i+1}^{i+2}|\gamma_{1}(t)-\gamma_{2}(t)|dt+ \frac{1}{2}\sum_{i=-k}^{-1}\frac{1}{2^{-i-1}}\int_{i+1}^{i+2}|\gamma_{1}(t)-\gamma_{2}(t)|dt \right)=\\
	&= \lim\limits_{k\to\infty}\left(2\sum_{j=1}^{k+1}\frac{1}{2^{j}}\int_{j}^{j+1}|\gamma_{1}(t)-\gamma_{2}(t)|dt+ \frac{1}{2}\sum_{i=-k+1}^{0}\frac{1}{2^{-j}}\int_{j}^{j+1}|\gamma_{1}(t)-\gamma_{2}(t)|dt \right)\leqslant\\
	&\leqslant 2\lim\limits_{k\to\infty}\left( \sum_{j=1}^{k+1}\frac{1}{2^{j}}\int_{j}^{j+1}|\gamma_{1}(t)-\gamma_{2}(t)|dt+ \sum_{j=-k+1}^{0}\frac{1}{2^{-j}}\int_{j}^{j+1}|\gamma_{1}(t)-\gamma_{2}(t)|dt \right)=\\
	&\leqslant 2 \sum_{j=-\infty}^{\infty}\frac{1}{2^{|j|}}\int_{j}^{j+1}|\gamma_{1}(t)-\gamma_{2}(t)|dt=2\rho(\gamma_1,\gamma_2).
	\end{align*}
	
	Hence $ F_{1} $ is continuous. The proof of continuity of the inverse is analogous. 
\end{proof}

\begin{definition}\label{definition entropy nsvf}
	In the conditions above, we define the topological entropy $ h $ of $ Z=(X,Y) $ on $ M $, as the topological entropy of the map $ F_{1}  $ in $ \Omega $, that is, $ h(Z):=h(F_{1} ) $.
\end{definition}

Note that, if we consider $ Z=(X,X) $ then $ Z $ is a smooth vector field. In this case, the value given by Definition \ref{definition entropy nsvf} should coincide with the value obtained in Definition \ref{definition entropy smooth}. The next proposition assures it.

\begin{proposition}\label{conjugacy smooth not smooth}
	Let $ Z=(X,X) $ a NSVF defined on $ M $, then the function $ f:M\to\Omega $ defined by $ f(x)=\gamma_{x} $ is a homeomorphism and the following diagram commutes:
	
	\begin{figure}[H]
		\begin{center}
			\begin{tikzpicture}
			\node (A) {$(M,d)$};
			\node (B) [right of=A] {$(M,d)$};
			\node (C) [below of=A] {$(\Omega,\rho)$};
			\node (D) [below of=B] {$(\Omega,\rho)$};
			\large\draw[->] (A) to node {\mbox{{\footnotesize $X^{1}$}}} (B);
			\large\draw[->] (C) to node {\mbox{{\footnotesize $F_{1} $}}} (D);
			\large\draw[->] (A) to node {\mbox{{\footnotesize $ f $}}} (C);
			\large\draw[->] (B) to node {\mbox{{\footnotesize $ f $}}} (D);
			\end{tikzpicture}
		\end{center}
		\caption{}
		\label{diagram1}
	\end{figure}
\end{proposition}

\begin{proof}
	Let $ \varphi $ denote the flow of $ X $. And define $ f:M\to\Omega $ as $ f(x)=\gamma_{x} $, where $ \gamma_{x}(t)=\varphi(t,x) $ for all $ t $. From Existence and Uniqueness Theorem, the function $ f $ is well-defined and it has inverse $ f^{-1}(\gamma)=\gamma(0) $.
	
	Let $ x_{1}, x_{2}\in M $, $ \gamma_{x_{1}}=f(x_{1}), \gamma_{x_{2}}=f(x_{2}) $. Since $ \varphi $ is smooth on $ M $, there exists a constant $ k_{1} $ such that $ |\gamma_{x_{1}}(t)-\gamma_{x_{2}}(t)|\leqslant k_{1}|x_{1}-x_{2}|$, for all $ t $. This implies that if $ |x_{1}- x_{2}|<\varepsilon $, then $ \int_{i}^{i+1}|\gamma_{x_{1}}(t)-\gamma_{x_{2}}(t)|dt<k_{1}\varepsilon $. 
	
	Hence
	\[ \rho(\gamma_{x_{1}},\gamma_{x_{2}})=\sum_{i\in\Z} \frac{1}{2^{|i|}}\int_{i}^{i+1}|\gamma_{x_{1}}(t)-\gamma_{x_{2}}(t)|dt<3k_{1}\varepsilon. \]
	
	Thus $ f $ is continuous.
	
	Now, by Proposition \ref{lemma-epsilon-proximos}, given $ \varepsilon>0 $, there is $ \delta>0 $, such that if $ \rho(\gamma_{1},\gamma_{2})<\delta $ then $ \gamma_{1}(0)=x_{1} $ and $ \gamma_{2}(0)=x_{2} $ are $ \varepsilon- $close, which proves that $ f^{-1} $ is continuous.
	
	Moreover, for all $ x\in M $ and all $ t\in\R $, $ f(X^{1}(x))(t)=f(\varphi(1,x))(t)=\gamma_{\varphi(1,x)}(t)=\gamma_{x}(t+1)=F_{1} (\gamma_{x})(t)= F_{1} (f(x))(t) $. Hence 
	\[ f\circ X^{1}=F_{1} \circ f. \]
	%
	%
	%
	%
\end{proof}

\begin{corollary}
	The entropy of $ Z=(X,X) $ (Definition \ref{definition entropy nsvf}) and entropy of $ X $ (Definition \ref{definition entropy smooth}) are the same.
\end{corollary}
\begin{proof}
	Propositions \ref{invariant conjugacy} and \ref{conjugacy smooth not smooth}. 
\end{proof}

\begin{theorem}\label{thm-closed-invariant}
	Let $ Z $ be a NSVF over $ M $, and $ K\subset M $ a closed invariant set for $ Z $. Then $ h(Z)\geqslant h(Z\vert_{K}) $.
\end{theorem}
\begin{proof}
	Let $ \Omega $ represent the set of all trajectories of $ Z $ over $ M $. Since $ K\subset M $ is closed and $ Z- $invariant, the set $ \Omega_{K}\subset \Omega $, of all trajectories through points of $ K $ is invariant for $ F_{1} $. Moreover, given any sequence $ \left(\gamma_{n}\right)\subset\Omega_{K} $, such that $ \gamma_{n}\to\gamma $, we have that, $ \gamma\in\Omega_{K} $, for if it is not, then $ \gamma\subset M\setminus K $, which is open, then for every point $ x $ of $ \gamma $, there is $ \varepsilon_{x} $ such that $ B(x,\varepsilon_{x})\subset M\setminus K $, which implies $ \gamma $ cannot be limit of trajectories of $ \Omega_{K} $. Thus $ \Omega_{K}\subset \Omega $ is closed.
	
	Now, $ \Omega_{K}\subset \Omega $ is closed and $ F_{1}- $invariant, then $ h(Z_{K})=h(F_{1}\vert_{\Omega_{K}})\leqslant h(F_{1})=h(Z) $, by Proposition \ref{prop-closed-f-invariant}.
\end{proof}

\section{Finite Positive Topological Entropy }

In this section we prove results about the existence of NSVFs of positive topological entropy, but finite.

\begin{theorem}\label{thm-entropy-log-alpha}
	Given $ \alpha\in\Z, \alpha\geqslant 2 $ there exists NSVF $ Z $ such that $ h(Z)=\log\alpha $.
\end{theorem}
\begin{proof}
	Consider the polynomial NSVF $ Z=(X,Y) $ with $ \Sigma=\{y=0\} $, where $ X=(1, 1-x) $ and $ Y=(-1,1-x) $. They are both symmetric and the point $ (1,0) $ is an invisible-invisible two-fold, and there is a closed trajectory that goes through the origin $ (0,0) $.
	
	Let $ \alpha=3 $, and consider six rays from the origin (one at each multiple of $ \frac{\pi}{3} $). In that way, the plane is divided into six different regions. Number each region from 1 to 6, counter clockwise. We define $ X $ in region 1, and $ Y $ in region 6. Now, define a vector field in each one of these regions, such that in regions 3 and 5 we have a phase portrait that is the one of $ X $ rotated, and in regions 2 and 4 the phase portrait is symmetric to the one of $ Y $ (see Figure \ref{rosacea-3}). Now there are three closed arcs that goes through the origin. That defines a NSVF with six different regions such that, apart from three invisible two-folds, and the origin, every other point is either regular or crossing, and every trajectory that does not goes through the origin is closed.
	
	\begin{figure}[h]
		\begin{overpic}[width=.6\linewidth]{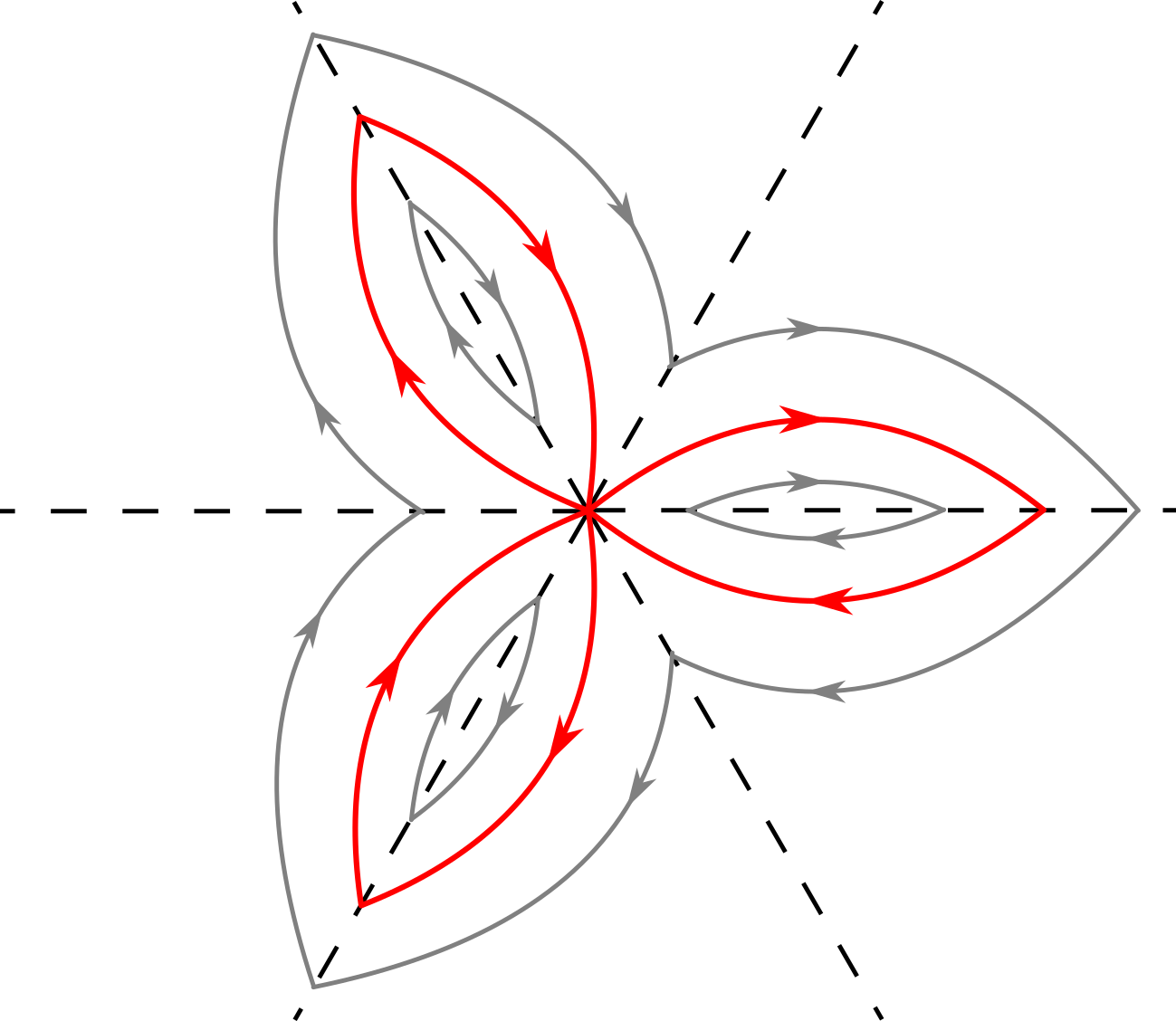}
			\put(89,40){\footnotesize $ I_{1} $} 
			\put(26,70){\footnotesize $ I_{2} $} 
			\put(35,8){\footnotesize $ I_{3} $}
		\end{overpic}
		\caption{Set highlighted in red is invariant for $ Z $}\label{rosacea-3}
	\end{figure}

	The union of the curves $ I_{1},I_{2},I_{3} $ highlighted in Figure \ref{rosacea-3} is invariant to the NSVF and a trajectory in it is the amalgamation of these three different arcs in every possible combination.
	
	Let $ \Omega^{*}=\{\gamma \in \Omega \,\,|\,\, \gamma(0)=0 \} $. Observe that the function $ F_{1} :\Omega^{*}\to\Omega^{*} $ is well defined since the arcs $ I_{j} $ can be adjusted so it has time duration $ 1 $.
	
	For any $ \gamma,\psi \in\Omega^{*} $ and for all $ i\in\Z $, we have that $ \gamma(i)=\psi(i)=0 $, and the integral 
	\[ \int_{i}^{i+1}|\gamma_{1}(t)-\gamma_{2}(t)|dt \]
	vanishes when the same arc is used at the same time in the construction of $ \gamma $ and $ \psi $ or it equals a constant $ \mu $, if different arcs occur (because of the symmetry of phase portrait this value $ \mu $ is the same for any pair $ I_{j},I_{k} $ with $ j\neq k $).
	
	Let $ \gamma_{j} $ represents a trajectory starting with $ I_{j} $, for $ j=1,2,3 $. Consider $ \varepsilon=2\mu $. The open ball $ B(\gamma_{j},\varepsilon,1) $ contains all trajectories that begins with one arc $ I_{j} $.
	Then $ \Omega^{*}\subset B(\gamma_{1},2\mu,1) \cup B(\gamma_{2},2\mu,1) \cup B(\gamma_{3},2\mu,1)$, which implies that $ S_{\rho}(F_{1} ,2\mu,1)\leqslant 3 $.
	
	Since $ \max\{\rho(\gamma,\psi), \gamma,\psi\in\Omega^{*}\}=3\mu $, it does not occur $ S_{\rho}(F_{1} ,2\mu,1)=1 $, then suppose by contradiction that $ S_{\rho}(F_{1} ,2\mu,1)=2 $ and let $ \psi_{1},\psi_{2}\in\Omega^{*} $ such that 
	\[ \Omega^{*}\subset B(\psi_{1},2\mu,1) \cup B(\psi_{2},2\mu,1). \] 
	Since there are three options $ I_{j} $, it is possible to construct $ \psi^{*} $ that does not coincide with either $ \psi_{1} $ or $ \psi_{2} $ at any given time. So, $ \rho(\psi^{*},\psi_{1})=\rho(\psi^{*},\psi_{2})=3\mu $ and, therefore,
	\[ \Omega^{*}\not\subset B(\psi_{1},2\mu,1) \cup B(\psi_{2},2\mu,1). \] 
	
	Thus $ S_{\rho}(F_{1} ,2\mu,1)= 3 $.
	
	Now fix $ \gamma_{11} $ a trajectory that begins with the concatenation of $ I_{1} $ twice in a row; $ \gamma_{12} $ for one starting with $ I_{1} $ and followed by $ I_{2} $ and so on. In general, fix $ \gamma_{jk} $ for a trajectory beginning with the concatenation of $ I_{j} $ followed by $ I_{k} $, for $ j,k=1,2,3 $.
	
	Given any $ \psi\in\Omega^{*} $, it coincides with some $ \gamma_{jk} $ in the interval $ [0,2] $. From which, we get that $ \rho(\gamma_{jk}, \psi)< 2\mu $ and $ \rho(F_{1}(\gamma_{jk}),F_{1}(\psi))<2\mu $, then $ \rho_{2}^{F_{1}}(\gamma_{jk},\psi)<2\mu $. Thus,
	\[ \Omega^{*}\subset\bigcup_{j,k=1,2,3} B(\gamma_{jk},2\mu,2), \]
	hence $ S_{\rho}(F_{1} ,2\mu,2)\leqslant 9 $.
	
	Once again, suppose for contradiction, that $ S_{\rho}(F_{1} ,2\mu,2)< 9 $. Then there exist at most $ 8 $ trajectories $ \psi_{l} $, such that 
	\begin{equation}\label{eq:e-capacidade-9}
		\Omega^{*}\subset\bigcup_{l} B(\psi_{l},2\mu,2).
	\end{equation}
	
	There exists a trajectory $ \psi^{*} $ that does not coincide with any $ \psi_{l} $ at the interval $ [-1,2] $. And, for any $ l $,
	\[\rho(\psi^{*},\psi_{l}) \geqslant \dfrac{1}{2}\int_{-1}^{0}|\psi^{*}(t)-\psi_{l}(t)|dt + \int_{0}^{1}|\psi^{*}(t)-\psi_{l}(t)|dt+ \dfrac{1}{2}\int_{1}^{2}|\psi^{*}(t)-\psi_{l}(t)|dt= 2\mu.
	\]
	
	Then $ \rho_{2}^{ F_{1}}(\psi^{*},\psi_{l})\geqslant 2\mu $, which contradicts (\ref{eq:e-capacidade-9}). Hence $ S_{\rho}(F_{1} ,2\mu,2) = 9 $.
	
	In general, for $ n $ fixed, there are $ 3^{n} $ fixed trajectories $ \gamma_{j_{0}\dots j_{n-1}} $, $ j_{i}\in\{1,2,3\} $, for which the open balls $ B_{F_{1}}(\gamma_{j_{0}\dots j_{n-1}},2\mu,n) $ cover $ \Omega^{*} $. Thus $ S_{\rho}(F_{1} ,2\mu, n)= 3^{n} $.
	
	Now consider $ \varepsilon=\dfrac{\mu}{2^{m}} $. For $ n=1 $ (metric $ \rho_{1}^{F_{1}} $), we choose trajectories $ \gamma_{j_{-m}\dots j_{m}} $ to be the centers of the open balls. And for a fixed $ n $, we choose trajectories $ \gamma_{j_{-m}\dots j_{m+n-1}} $. Then $ S_{\rho}(F_{1} ,\dfrac{\mu}{2^{m}}, n)= 3^{2m+n} $.
	
	So,
	\[ h(F_{1})=\lim_{m\to\infty}\lim_{n\to\infty}\dfrac{1}{n}\log S_{\rho}(F_{1} ,\dfrac{\mu}{2^{m}}, n)= \lim_{m\to\infty}\lim_{n\to\infty}\dfrac{1}{n}\log 3^{2m+n} = \log 3.  \]
	
	Now, for any $ \alpha\in\Z, \alpha\geqslant 2 $, we can construct a NSVF analogous to the one above (with $ 2\alpha $ regions) and, by the same calculations, conclude $ h(F_{1})=\log\alpha $.
\end{proof}

Consider the figure with three arcs (Figure \ref{rosacea-3}), in demonstration of Theorem \ref{thm-entropy-log-alpha}. If we consider $ \widetilde{Z} $ the restriction of $ Z $ to $ I_{1}\cup I_{2} $ (the union of only two arcs), then $ h(\widetilde{Z})=\log 2 $.

\begin{corollary}
	More generally, given $ Z $ with $ 2\alpha $ regions, as in proof of Theorem \ref{thm-entropy-log-alpha} (with $ \alpha $ closed arcs), we can consider $ \widetilde{Z} $ as the restriction of $ Z $ to any number $ \beta<\alpha $ of closed arcs, and then $ h(\widetilde{Z})=\log\beta $.
\end{corollary}

\section{Infinite Topological Entropy}
In this section, we prove the existence of a NSVF of infinite topological entropy with an example, and next we give sufficient conditions for a NSVF $ Z $ to have $ h(Z)=\infty $.
 
\begin{example}\label{exemplo-feijao}
	Consider the NSVF:
	\begin{equation}\label{Eq Campo feijao}
	Z(x,y)=\left\{
	\begin{array}{l} 
	X(x,y)= (1, -2x) ,\quad $for$ \quad y \geq 0 \\ 
	Y(x,y)= (-2,-4x^3+2x),\quad $for$ \quad y \leq 0
	\end{array}
	\right. .
	\end{equation}
	
	\begin{figure}[h]
		\begin{overpic}[width=.65\linewidth]{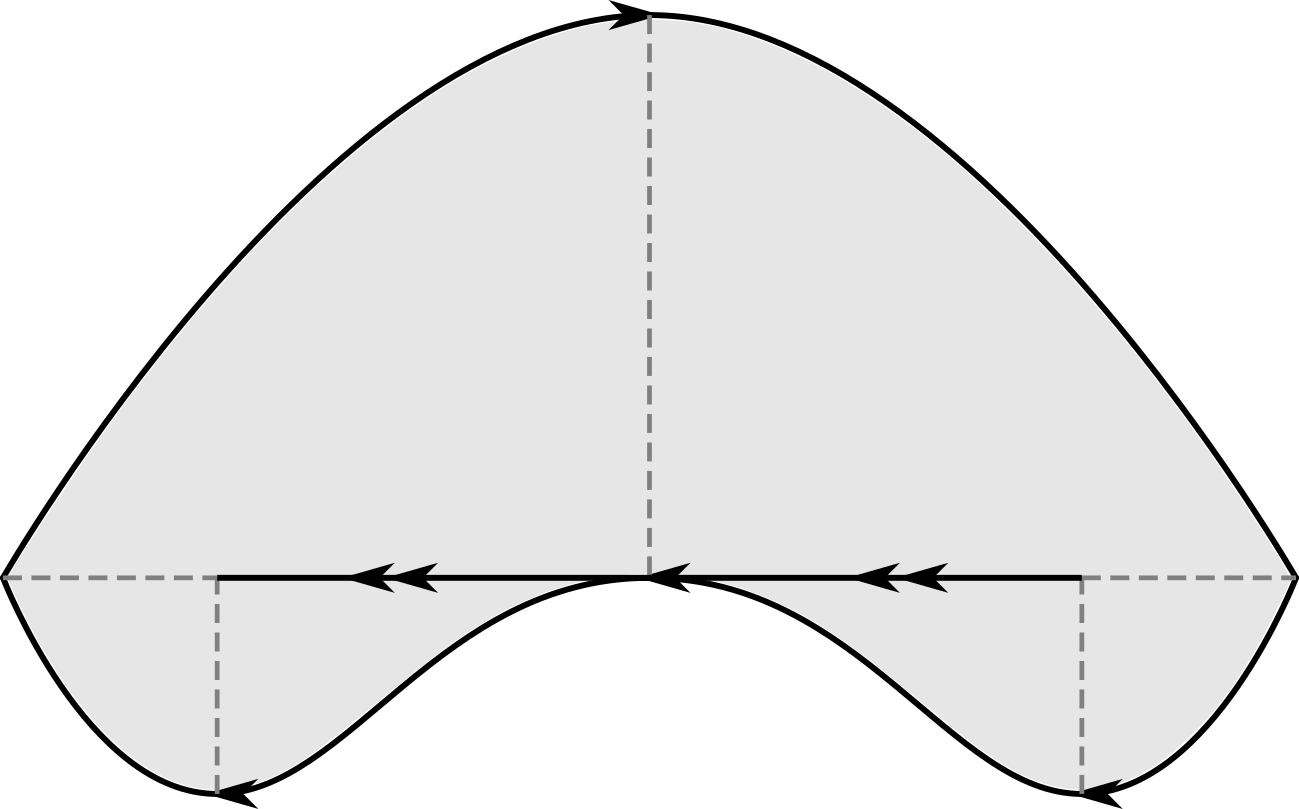}
			\put(50,13){\footnotesize $ p $} 
		\end{overpic}
		\caption{Set $ K $ is invariant for $ Z $}\label{figurafeijao1}
	\end{figure}
	
	It has a compact invariant set $K=\{(x,y)\in\R^2:-1\leqslant x\leqslant 1$ and $x^4/2-x^2/2\leqslant y\leqslant 1-x^2 \}$. Furthermore, $K$ is a chaotic non-trivial minimal set for $Z$ (see \cite{BCEminimal}).  As before, we will consider $ \Omega=\{\gamma $ global trajectory of $ Z\vert_{\Lambda}\} $ the set of all possible trajectories contained in the set $ \Lambda $ (see Figure \ref{figurafeijao1}).

	We will show that $ h(Z)=\infty $. 
	
	
	\begin{definition}\label{def:escape-point}
		We say that a trajectory $ \gamma $ escapes $ \overline{\Sigma^{e}} $ at a time $ s_{0} $ if $ \gamma(s_{0})\in\overline{\Sigma^{e}} $ and there is $ s_{1}>s_{0} $, such that $ \gamma(s)\notin\overline{\Sigma^{e}} $, for all $ s_{0}<s<s_{1} $. We may also add that $ \gamma $ escapes to $ \Sigma^{+} $ or $ \Sigma^{-} $, depending on which of these regions the points $ \gamma(s) $ are located in. Alternatively, we may say that $ \gamma(s_{0}) $ is a \emph{escape point of $\gamma$}. 
	\end{definition}
	
	Since every trajectory of $ \Omega $ intersects $ \Sigma^{e} $ in positive time, define the function $ \tau:\Omega\to \R_{+}, \tau(\gamma)=\tau_{\gamma} $, as the minimum $ t>0 $, for which $ \gamma $ escapes $ \overline{\Sigma^{e}} $, in other words, the time $ \tau_{\gamma} $ is the ``next escape time'' for $ \gamma $ (see Figure \ref{figurafeijao3}).
	
	\begin{figure}[h]
		\begin{overpic}[width=.65\linewidth]{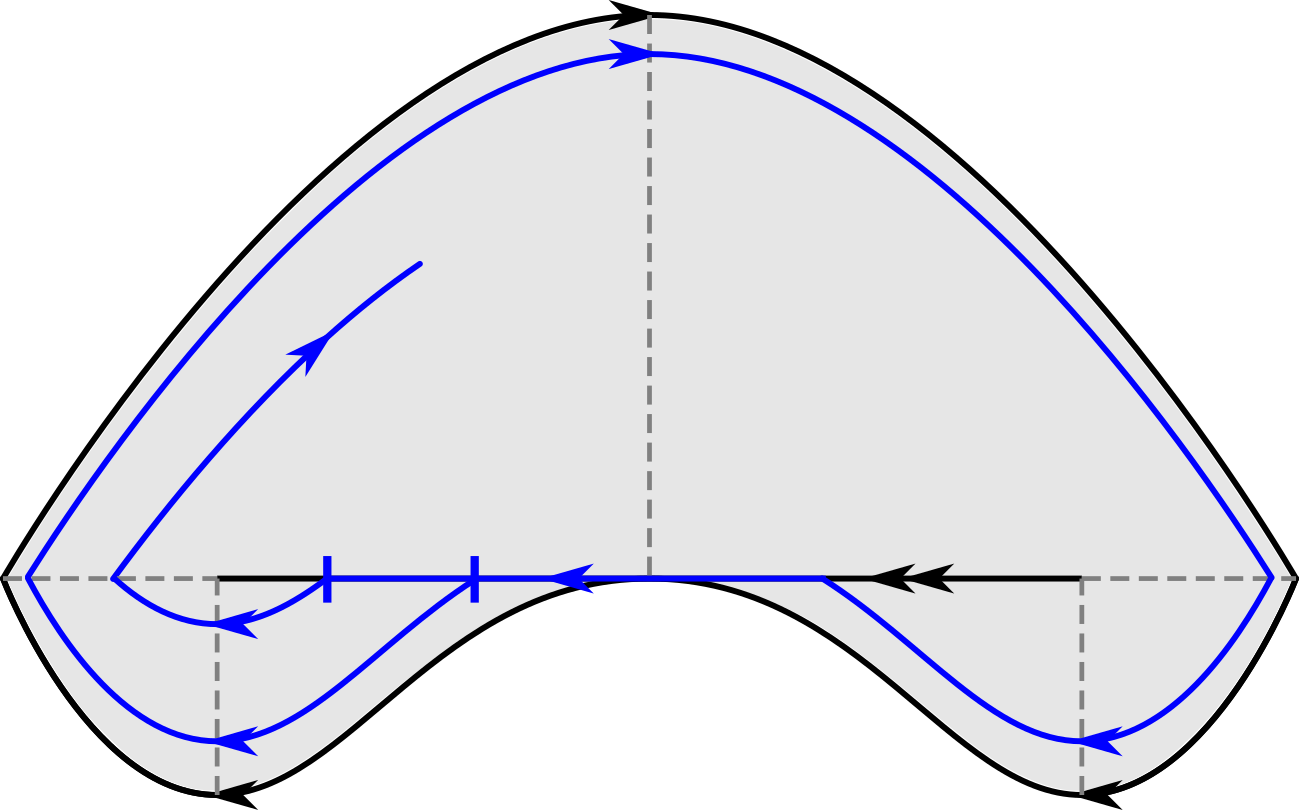}
			\put(50,13){\footnotesize $ p $} 
			\put(35,21){\footnotesize $ x $} 
			\put(25,21){\footnotesize $ y $} 
		\end{overpic}
		\caption{Arc of trajectory $ \gamma $, such that, $ \gamma(0)=x $, and $ \gamma(\tau_{\gamma})=y $. Both $ x $ and $ y $ are escape points of $\gamma$.
		} \label{figurafeijao3}
	\end{figure}
	
	
	Now, consider a subset $ R\subset K $, bounded by two distinct simple closed trajectories, and let 
	\[ S=\{\gamma\in\Omega \,\,|\,\, \forall t, \gamma(t)\in R \text{ and } \gamma \text{ escapes to } \Sigma^{-} \text{ at } t=0 \} \] 
	that is, $ S $ is the set of trajectories entirely contained in $ R $ that escapes at $ s_{0}=0 $ (see Figure \ref{figurafeijao2}).
	We may restrict $ R $ in order to obtain some positive integer $ k>1 $ for which we have $ k-1<\tau(\gamma)<k $ for every $ \gamma\in S $.
	And define the map:
	\begin{align*}
	P:S&\to S\\
	\gamma&\mapsto P(\gamma),
	\end{align*}
	where $ P(\gamma)(\ponto)=\gamma(\ponto+\tau_{\gamma}) $.

	\begin{figure}[h]
		\begin{overpic}[width=.65\linewidth]{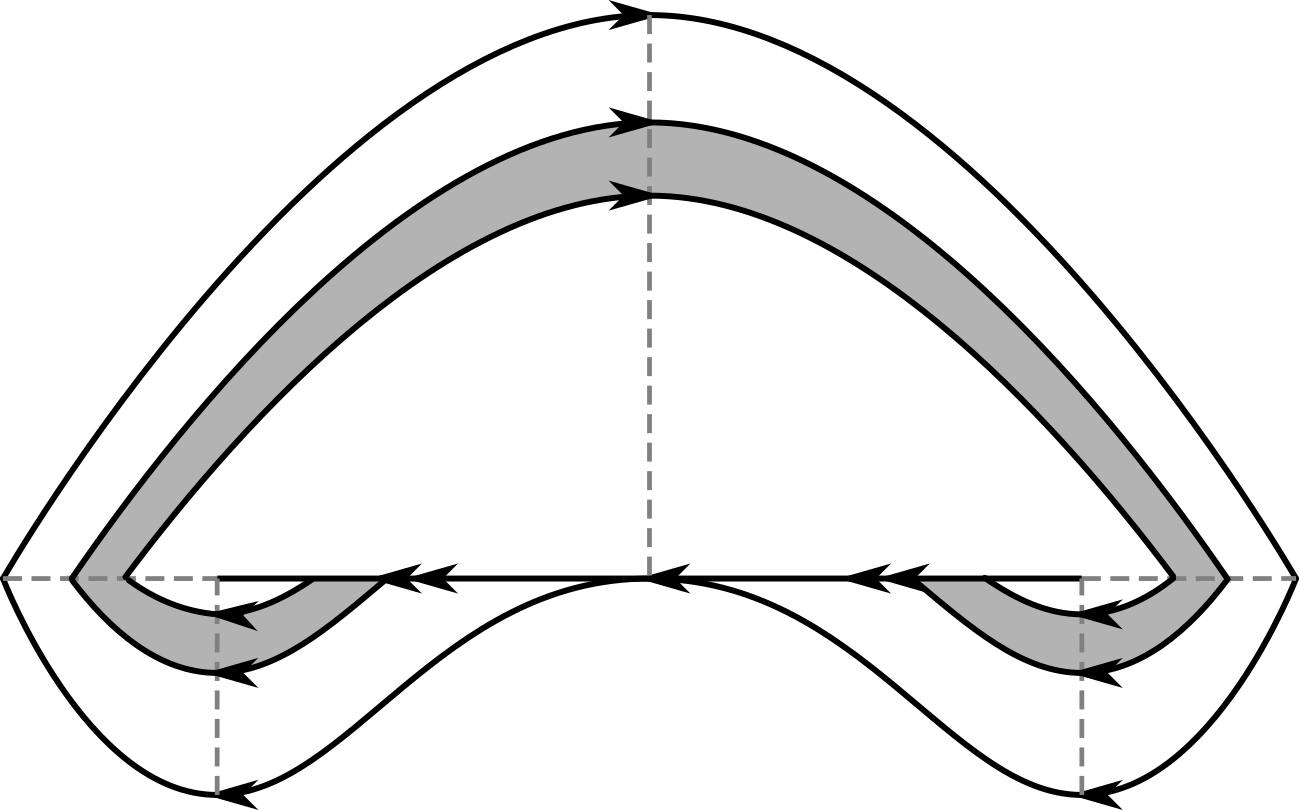}
			\put(50,13){\footnotesize $ p $} 
			\put(69,40){\footnotesize $ R $} 
		\end{overpic}
		\caption{Region $ R $}\label{figurafeijao2}
	\end{figure}
	




	In the following, we prove some properties of $ \tau $ and $ P $, that will be useful next.
	
	\begin{proposition}\label{lemma-tau-continua}
		The function $ \tau:S\to \R_{+} $ is continuous.
	\end{proposition}
	\begin{proof}
		
		%
		%
		%
		%
		%
		%

		Let $ \gamma_{1},\gamma_{2}\in S, \varepsilon>0 $ and consider $ t_{0}>k $ large enough. For simplicity, let $ \tau_{1}=\tau_{\gamma_{1}} $ and $ \tau_{2}=\tau_{\gamma_{2}} $ and assume $ \tau_{1}<\tau_{2} $. 
		
		By Proposition \ref{lemma-epsilon-proximos}, there is $ \delta>0 $ such that, for any $ \gamma_{1},\gamma_{2}\in S $, if $ \rho(\gamma_{1},\gamma_{2})<\delta $, then for every $ |t|<t_{0} $, $ |\gamma_{1}(t)-\gamma_{2}(t)|\leqslant \varepsilon $. This implies that $ \gamma_{1}(\tau_{1}) $ and $ \gamma_{2}(\tau_{1}) $ are $ \varepsilon $ close to each other and also $ \gamma_{1}(\tau_{2}) $ and $ \gamma_{2}(\tau_{2}) $.
		
		Now, there exists $ \overline{t}>0 $, that is the minimum time for which $ B(\gamma_{1}(\tau_{1}+\overline{t}),\varepsilon)\cap \overline{\Sigma^{e}}=\emptyset $. That is, the time required to $ \gamma_{1} $ distance $ \varepsilon $ from $ \Sigma^{e} $. It clearly depends continuously of $ \varepsilon $.
		
		Since $ |\gamma_{1}(\tau_{2})-\gamma_{2}(\tau_{2})|<\varepsilon $, and $ \gamma_{2}(\tau_{2})\in\Sigma^{e} $, we must have $ \tau_{1}<\tau_{2}<\tau_{1}+\overline{t} $.
		
		Intuitively, from the moment $ \gamma_{1} $ escapes $ \Sigma^{e} $, there is little time to $ \gamma_{2} $ escape it too, for if it does not, they would distance themselves.

		%
		%
		%
		%
		%
		%
	\end{proof}

	\begin{proposition}\label{prop:p-continuous}
		The map $ P:S\to S $ as defined above is continuous.
	\end{proposition}
	\begin{proof}
		Let $ \gamma\in S $ and $ \gamma_{n}\to\gamma $. We will show that $ P(\gamma_{n})\to P(\gamma) $.

		Let $ \varepsilon>0 $ then there is $ t_{0}>0 $ such that \[ \sum_{|i|\geqslant t_{0}}\frac{1}{2^{|i|}}\int_{i}^{i+1}| P(\gamma_{n})(t)- P(\gamma)(t)|dt <\frac{\varepsilon}{2}. \]
		
		Let $ \varepsilon_{0}>0 $ such that
		\[ \forall |t|<t_{0}, |P(\gamma_{n})(t)-P(\gamma)(t)|<\varepsilon_{0} \Rightarrow \sum_{|i|< t_{0}}\frac{1}{2^{|i|}}\int_{i}^{i+1}| P(\gamma_{n})(t)- P(\gamma)(t)|dt <\frac{\varepsilon}{2}. \]
		
		Now, by Propositions \ref{lemma-converg-pontual} and \ref{lemma-tau-continua}, and considering the fact that $ \gamma $ is piecewise differentiable with bounded derivatives, we have that, there is $ N $ large enough such that if $ n>N $:
		\begin{align*}
		|P(\gamma_{n})(t)-P(\gamma)(t)|&=|\gamma_{n}(t+\tau_{\gamma_{n}})-\gamma(t+\tau_{\gamma})|\leqslant\\
		&\leqslant  |\gamma_{n}(t+\tau_{\gamma_{n}})-\gamma(t+\tau_{\gamma_{n}})|+ |\gamma(t+\tau_{\gamma_{n}})-\gamma(t+\tau_{\gamma})|<\\
		&<\frac{\varepsilon_{0}}{2}+M |\tau_{\gamma_{n}}-\tau_{\gamma}|<\varepsilon_{0}.
		\end{align*}
		
		Hence, if $ n>N $, $ \rho(P(\gamma_{n}),P(\gamma))<\varepsilon $.
		
	\end{proof}

	\begin{proposition}\label{prop:p-entropia-infinita}
		The map $ P $ has infinite topological entropy.
	\end{proposition}
	\begin{proof}
		Consider a linear parametrization $\theta: \overline{\Sigma^{e}}\cap R \to [0,1] $. Then construct a function itinerary $ s:S\to [0,1]^{\Z} $ given by: $s(\gamma)=\left(s_{j}(\gamma)\right)_{j\in\Z} $, such that $ s_{0}(\gamma)=\theta(\gamma(0)) $, and $ s_{j}(\gamma)=\theta(P^{j}(\gamma)(0)) $.
		
		Now, let $ (b_{j})_{j\in\Z}=\sigma\circ s (\gamma) $, and $ (a_{j})_{j\in\Z}=s\circ P (\gamma) $, where $ {\sigma:[0,1]^{\Z}\to [0,1]^{\Z} }$ is the shift map. Then:
		\[ b_{j}=s_{j+1}(\gamma)=\theta(P^{j+1}(\gamma)(0))=\theta(P^{j}(P(\gamma))(0))=s_{j}(P(\gamma))=a_{j}. \]
		
		Hence $ \sigma\circ s=s\circ P $.
		\begin{figure}[H]
			\begin{center}
				\begin{tikzpicture}
				\node (A) {$S$};
				\node (B) [right of=A] {$S$};
				\node (C) [below of=A] {$[0,1]^{\Z}$};
				\node (D) [below of=B] {$[0,1]^{\Z}$};
				\large\draw[->] (A) to node {\mbox{{\footnotesize $P$}}} (B);
				\large\draw[->] (C) to node {\mbox{{\footnotesize $\sigma $}}} (D);
				\large\draw[->] (A) to node {\mbox{{\footnotesize $ s $}}} (C);
				\large\draw[->] (B) to node {\mbox{{\footnotesize $ s $}}} (D);
				\end{tikzpicture}
			\end{center}
		\end{figure}
		
		The map $ s $ is continuous. In fact,
		\[ |s_{j}(\gamma_{1})-s_{j}(\gamma_{2})|=|\theta(P^{j}(\gamma_{1})(0))-\theta(P^{j}(\gamma_{2})(0))| = \theta'|P^{j}(\gamma_{1})(0)-P^{j}(\gamma_{2})(0)|. \]
		
		Given $ \varepsilon>0 $, there exists $ \eta>0 $ and $ N_{0}\in\N $, such that, if $ |s_{j}(\gamma_{1})-s_{j}(\gamma_{2})|<\eta $ for all $ -N_{0}<j<N_{0} $, then $ d(s(\gamma_{1}),s(\gamma_{2}))<\varepsilon $.
		
		But, since $ P $ is continuous, there exists $ \delta>0 $, such that $ \theta'|P^{j}(\gamma_{1})(0)-P^{j}(\gamma_{2})(0)|<\eta $, for all $ -N_{0}<j<N_{0} $. Hence $ s $ is continuous.
		
		By Proposition \ref{invariant conjugacy}, $ h(P)=\infty $.
	\end{proof}

	\begin{theorem}\label{thm:feijao-entropia-infinita}
		The NSVF $ Z $ presented in (\ref{Eq Campo feijao}) has infinite topological entropy.
	\end{theorem}
	\begin{proof}
		In order to prove this, we show that the following diagram commutes: 
		\begin{figure}[H]
			\begin{center}
				\begin{tikzpicture}
				\node (A) {$ S $};
				\node (B) [right of=A] {$ F_{1}^{k}(S) $};
				\node (C) [below of=A] {$ S $};
				\node (D) [below of=B] {$ S $};
				\large\draw[->] (A) to node {\mbox{{\footnotesize $F_{1}^{k}$}}} (B);
				\large\draw[->] (C) to node {\mbox{{\footnotesize $ P $}}} (D);
				\large\draw[->] (A) to node {\mbox{{\footnotesize $ P $}}} (C);
				\large\draw[->] (B) to node {\mbox{{\footnotesize $ P $}}} (D);
				\end{tikzpicture}
			\end{center}
		\end{figure}
		where $ k>1$ is the integer given above, such that $ k-1<\tau_{\gamma}<k $, for all $ \gamma\in S $. 
		
		Given $ \gamma\in S $, and $ t $, on one side we have	
		\[ P\circ P(\gamma)(t)=P(\gamma)(t+\tau_{P(\gamma)})=\gamma(t+\tau_{P(\gamma)}+\tau_{\gamma}) \]
		and on the other side
		\[ P\circ F_{1}^{k}(\gamma)(t)=F_{1}^{k}(\gamma)(t+\tau_{F_{1}^{k}(\gamma)})=\gamma(t+\tau_{F_{1}^{k}(\gamma)}+k). \]
		
		Then it is enough to show that, for any $ \gamma\in S $:
		\begin{equation}\label{eqn:equal-times}
		\tau_{P(\gamma)}+\tau_{\gamma}= k+ \tau_{F_{1}^{k}(\gamma)}. 
		\end{equation}

		In fact, on the one hand, $ \tau_{P(\gamma)}+\tau_{\gamma} $ is exactly the time at which $ \gamma $ escapes $ \overline{\Sigma^{e}} $ for the second time. And since $ k-1<\tau_{\gamma}<k $, we have $ k\leqslant 2k-2 < \tau_{P(\gamma)}+\tau_{\gamma}<2k. $
		
		That is, 
		\begin{equation}\label{eqn:left-side}
		\tau_{P(\gamma)}+\tau_{\gamma}= \min\{s>k\,\,|\,\, \gamma(s) \text{ is a escape point of } \gamma \}.
		\end{equation}
		On the other hand:
		\begin{align*}
		k+ \tau_{F_{1}^{k}(\gamma)}&= k+ \min\{t>0\,\,|\,\, F_{1}^{k}(\gamma)(t) \text{ is a escape point of } F_{1}^{k}(\gamma) \}=\\
		&= k+ \min\{t>0\,\,|\,\, \gamma(t+k) \text{ is a escape point of } \gamma \}=\\
		&= k+ \min\{s-k>0\,\,|\,\, \gamma(s) \text{ is a escape point of } \gamma \}=\\
		&= \min\{s>k\,\,|\,\, \gamma(s) \text{ is a escape point of } \gamma \}.
		\end{align*}
		
		From the above equation and (\ref{eqn:left-side}), we get (\ref{eqn:equal-times}).

		Since $ P $ is onto and continuous it works as a semi-conjugacy between $ F_{1}^{k} $ and $ P $, then by Proposition \ref{invariant conjugacy}, $ h(F_{1})\geqslant h(P)=\infty $. Thus $ h(Z)=\infty $.
		
	\end{proof}
	
\end{example}

In the following theorem, we give sufficient conditions for a NSVF $ Z $ to have infinity topological entropy, but first one lemma that will be useful to complete the theorem.

\begin{lemma}\label{lemma:rescaling-fields-same-entropy}
	Let $ Z $ and $ \widetilde{Z} $ two NSVF defined over a compact invariant set $ K\subset \R^2 $, such that $ Z=c \cdot \widetilde{Z} $, where $ c>0 $ is an integer. If $ h(\widetilde{Z})=\infty $, then $ h(Z)=\infty $.
\end{lemma}
\begin{proof}
	Let $ \gamma $ be a trajectory of $ Z $. It is not difficult to see that $ \widetilde{\gamma}(t)=\gamma(\frac{1}{c}t) $ is trajectory of $ \widetilde{Z} $. If $ \Omega $ and $\widetilde{\Omega} $ represent the spaces of trajectories of $ Z $ and $ \widetilde{Z} $, respectively, we can define a continuous bijection between them $ H:\Omega\to\widetilde{\Omega} $, simply by $ H(\gamma)(t)=\widetilde{\gamma}(t)=\gamma(\frac{1}{c}t) $.
	
	Consider the following diagram:
	\begin{figure}[H]
		\begin{center}
			\begin{tikzpicture}
			\node (A) {$ \Omega $};
			\node (B) [right of=A] {$ \Omega $};
			\node (C) [below of=A] {$ \widetilde{\Omega} $};
			\node (D) [below of=B] {$ \widetilde{\Omega} $};
			\large\draw[->] (A) to node {\mbox{{\footnotesize $F_{1}$}}} (B);
			\large\draw[->] (C) to node {\mbox{{\footnotesize $ \widetilde{F_{1}}^{c} $}}} (D);
			\large\draw[->] (A) to node {\mbox{{\footnotesize $ H $}}} (C);
			\large\draw[->] (B) to node {\mbox{{\footnotesize $ H $}}} (D);
			\end{tikzpicture}
		\end{center}
	\end{figure}
	
	The diagram above commutes. In fact, for any $ t\in\R $, we have:
	\[ H\circ F_{1}(\gamma)(t)=F_{1}(\gamma)(\frac{1}{c}t)=\gamma(\frac{1}{c}t+1). \]
	
	On the other hand:
	\[ \widetilde{F_{1}}^{c}\circ H(\gamma)(t)=H(\gamma)(t+c)=\gamma(\frac{1}{c}t+1). \]
	
	Therefore $ h(Z)=h(F_{1})=h(\widetilde{F_{1}}^{c})=\infty $.
\end{proof}

\begin{theorem}\label{thm:infinite-entropy}
	Let $ Z=(X,Y) $ be a NSVF defined over a compact invariant set $ K\subset \R^2 $, with $ \Sigma^{e}\neq\emptyset $. Assume there exists a closed interval $ J\subset\Sigma^{e} $, such that for all $ x\in J $, there is a trajectory $ \gamma_{x} $ that escapes through $ x $ to $ \Sigma^{+} $ (or $ \Sigma^{-} $) and $ \gamma_{x} $ intersects $ J $ in finite time (see Figure \ref{fig:entropia-infinita}). Then $ h(Z)=\infty $.
\end{theorem}
\begin{figure}[h]
	\begin{overpic}[width=.65\linewidth]{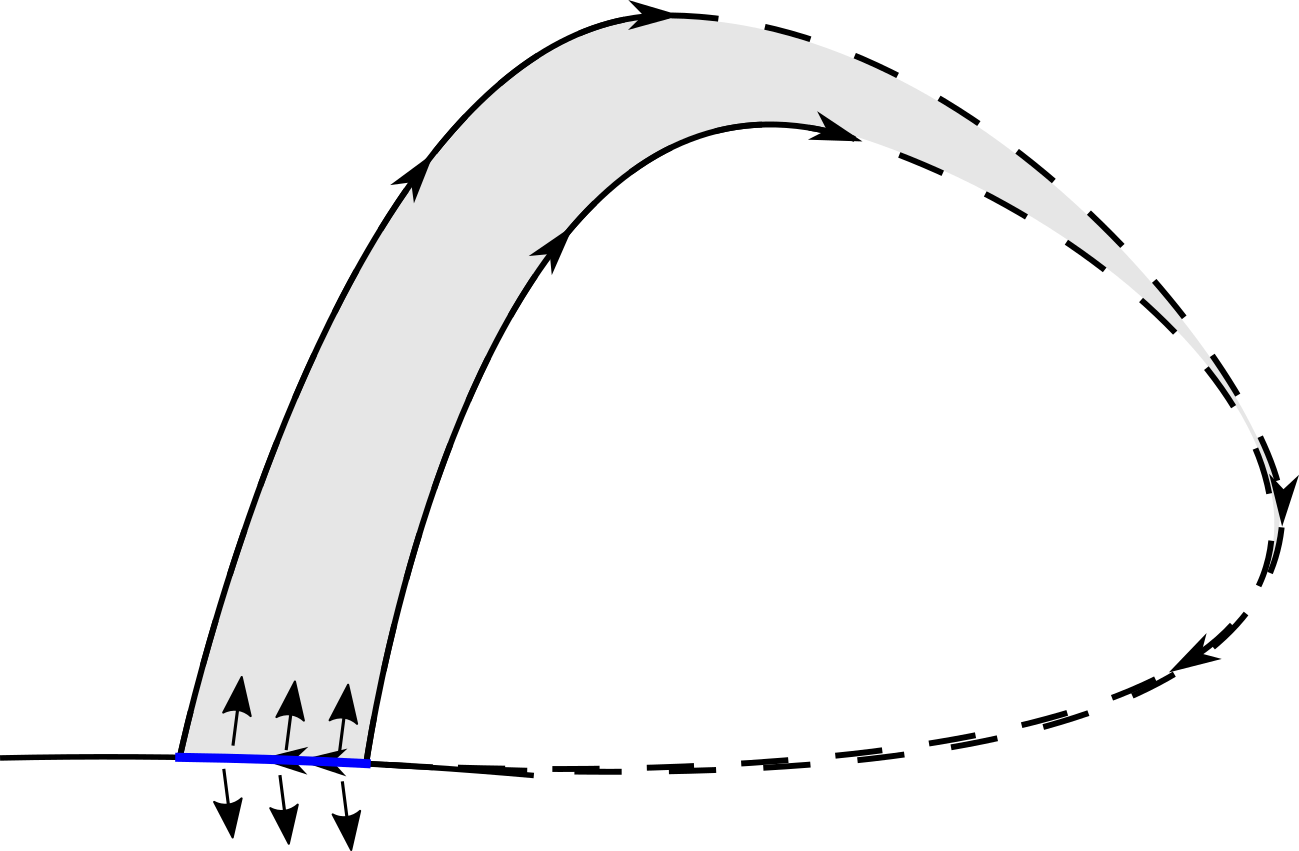}
		\put(2,11){\footnotesize $ \Sigma^{+} $} 
		\put(2,0){\footnotesize $ \Sigma^{-} $} 
		\put(-3.5,6){\footnotesize $ \Sigma $}
	\end{overpic}
	\caption{Every point of $ J $ (in blue) has a trajectory that escapes to $ \Sigma^{+} $ and returns to $ J $.} \label{fig:entropia-infinita}
\end{figure}

\begin{proof}
	From hypothesis, we get that for every $ x, y\in J $ there is an arc of trajectory escaping at $ x $ and going to $ y $. Consider $ S $ to be the set of every possible trajectory $ \gamma $ that is formed by a concatenation of these arcs and $ \gamma(0)\in J $ is a escape point. On this set $ S $, define the functions $ \tau $ and $ P $ in the same way as defined before. 
	
	Note that Propositions \ref{lemma-tau-continua} to \ref{prop:p-entropia-infinita} can be repeated word by word, since they do not depend on the vector field.
	
	In the example, we required an integer $ k>1 $ such that $ k-1<\tau_{\gamma}<k $ for all $ \gamma\in S $. If we can restrict to $ J_{0}\subset J $ in which this inequality holds, then the Theorem \ref{thm:feijao-entropia-infinita} can also be repeated which concludes the proof.
	
	If it is not the case, we must do some considerations first: since $ J $ is closed, $ \tau $ has a minimum value $ M>0 $, in this case less than 1, so let $ c>0 $ be the smallest integer such that $ 1<c\cdot M $ and define $ \widetilde{Z}=\frac{1}{c}Z $. Each one of the hypothesis are still true for $ \widetilde{Z} $, and now we can find an integer $ k $ satisfying that inequality. Hence $ h(\widetilde{Z})=\infty $ and the Lemma \ref{lemma:rescaling-fields-same-entropy} concludes the proof.
	
\end{proof}

\section{Examples}\label{examples}

\begin{example}
	Consider the NSVF $ Z=(X,Y) $ where $ X=\left(1,\frac{x}{2}-4x^3\right) $ and $ Y=\left(-1,\frac{x}{2}-4x^3\right) $ and $ \Sigma=\{y=0\} $, we have that $ h(Z)=\log 2 $. In fact, there is a set $ \Lambda $ that is $ Z- $invariant, that is the union of two  symmetric integral curves of $ X $ and $ Y $ (see Figure \ref{figura-8}) in which we can distinguish two arcs $ I_{1},I_{2} $ and apply the same calculations from the proof of Theorem \ref{thm-entropy-log-alpha}. Note that this example presents a polynomial NSVF that satisfies Theorem \ref{thm-entropy-log-alpha} for $ \alpha=2 $ with only two regions, using a visible two-fold.
	\begin{figure}[h]
		\begin{overpic}[width=.65\linewidth]{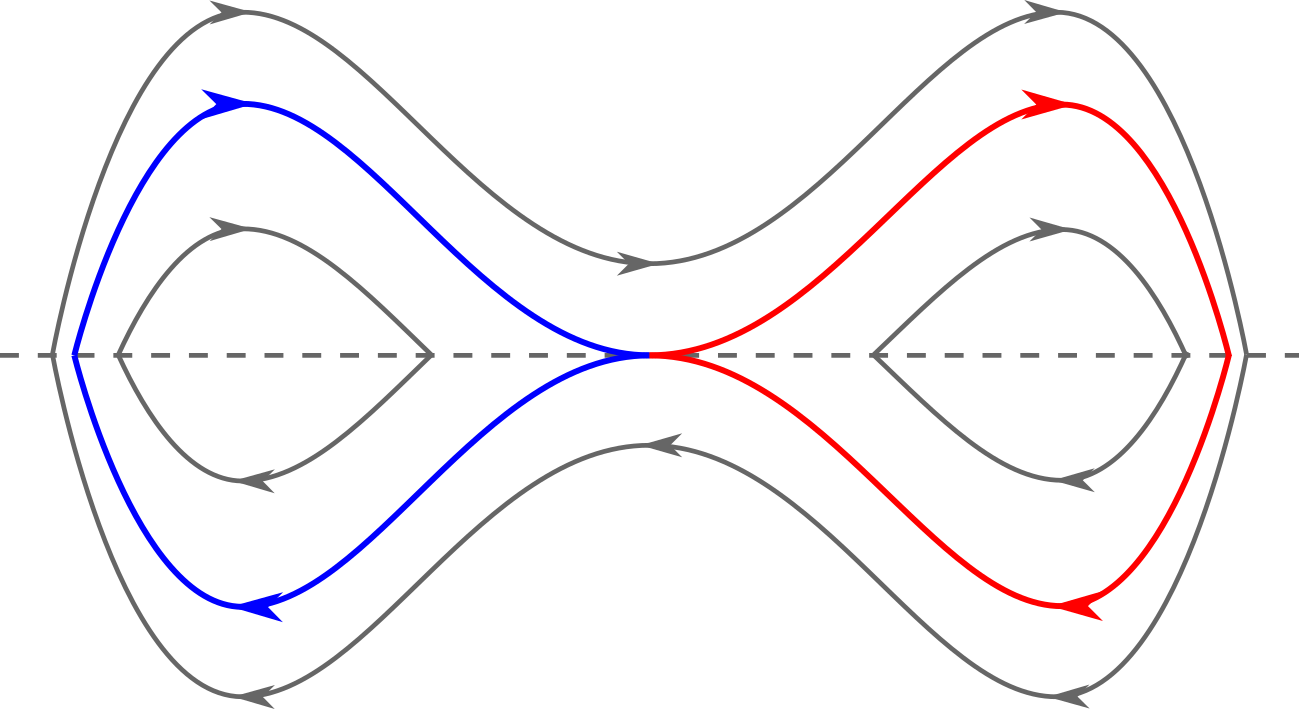}
		\end{overpic}
		\caption{Polynomial NSVF with just two regions and topological entropy equal $ \log 2 $}\label{figura-8}
	\end{figure}
\end{example}

\begin{example}
	In \cite{LF-HairyBall}, the authors present in details a NSVF $ Z $ tangent to $ \mathbb{S}^{2} $, without equilibrium points that has a non-trivial minimal set diffeomorphic to that presented in Example  \ref{exemplo-feijao} (see Figure \ref{fig:esfera-infinita}). Thus, by Theorem \ref{thm-closed-invariant}, $ h(Z)=\infty $.
\begin{figure}[h]
	\begin{overpic}[width=.5\linewidth]{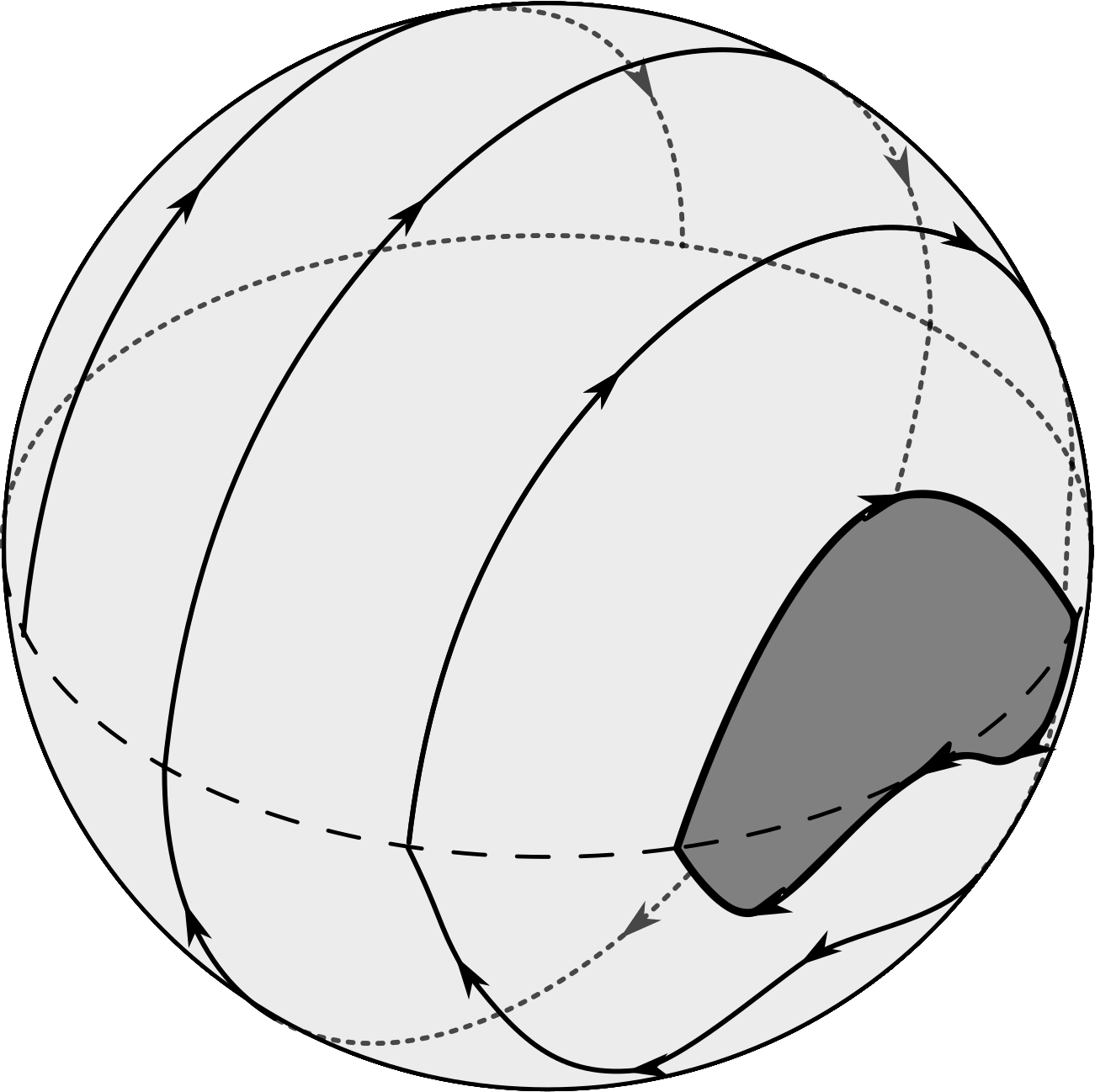}
	\end{overpic}
	\caption{NSVF tangent to $ \mathbb{S}^{2} $ with infinite topological entropy.}\label{fig:esfera-infinita}
\end{figure}

\end{example}

\begin{example}\label{exm:esfera-zero}
	Let $ Z $ be a NSVF defined over $ \mathbb{S}^{2} $, with discontinuity on $ \mathbb{S}^{1} $, and $ \Sigma^{+}, \Sigma^{-} $ are the northern and southern hemispheres. Consider $ X^{+} $ being a vector field without singularities, and $ X^{-} $ a vector field with a visible attracting node. Then it has a escaping region in some segment on $ \mathbb{S}^{1} $ and all global orbits tend to $ p $ in positive time. And in negative time it reaches the escaping region (see Figures \ref{fig:esfera-zero} and \ref{fig:esfera-zero-2}).
	\begin{figure}[h]
		\begin{overpic}[width=.5\linewidth]{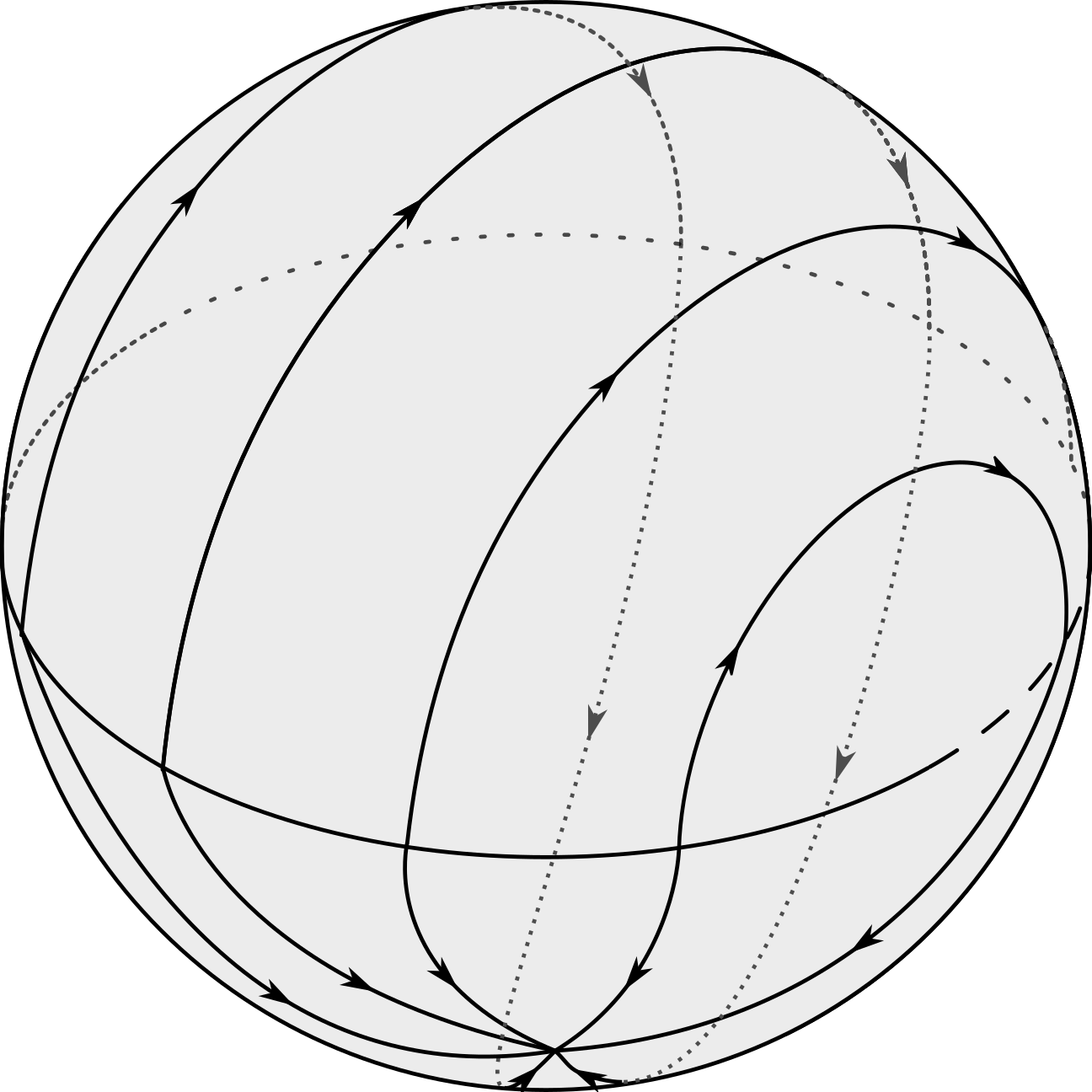}
		\end{overpic}
		\caption{NSVF tangent to $ \mathbb{S}^{2} $ with null topological entropy.}\label{fig:esfera-zero}
	\end{figure}

	\begin{figure}[h]
		\begin{overpic}[width=.8\linewidth]{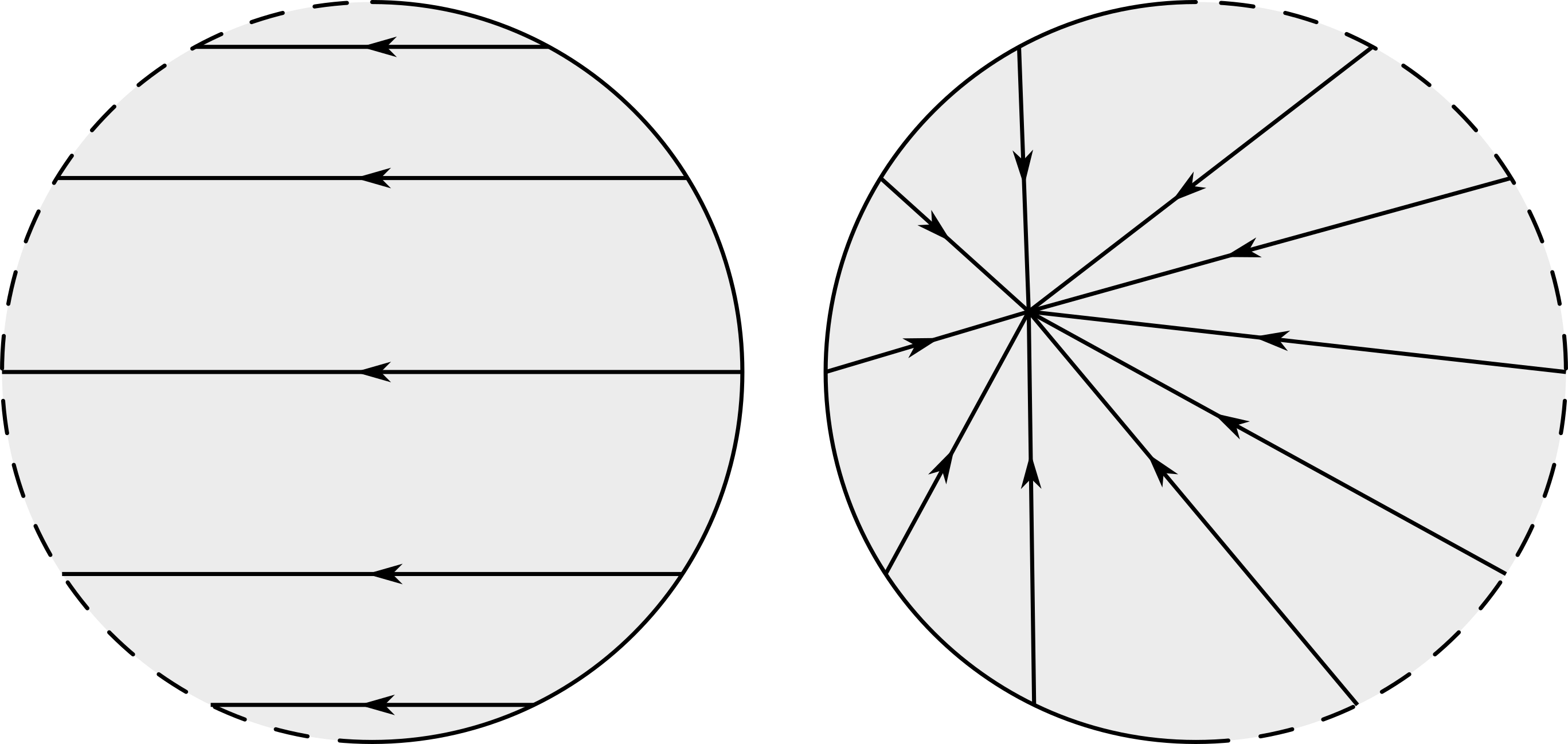}
		\end{overpic}
		\caption{Northern and southern hemispheres in Example \ref{exm:esfera-zero}.}\label{fig:esfera-zero-2}
	\end{figure}
	Now, for all $ \varepsilon>0 $, there exists $ m\in\N $ such that given any $ \gamma\in\Omega $, for any $ n>m $, $ |F_{1}^{n}(\gamma)(0)-p|<\varepsilon $, because every trajectory tends to $ p $. This implies that $ \rho(F_{1}^{n}(\gamma),\gamma_{p})< \varepsilon $.
	
	Then the $ \varepsilon $-capacity, $ S_{\rho}(F_{1},\varepsilon,n) $ does not grow with $ n $, then $ \limsup\frac{1}{n}\log S_{\rho}(F_{1},\varepsilon,n)=0 $. Hence $ h(Z)=0 $.

\end{example}

\section*{Acknowledgements}
A. A. Antunes is supported by grant \#2017/18255-6, São Paulo Research 
Foundation (FAPESP). T. Carvalho is partially supported by S\~{a}o Paulo Research Foundation (FAPESP grants 2019/10450-0 and 2019/10269-3) and by Conselho Nacional de Desenvolvimento Cient\'ifico e Tecnol\'ogico (CNPq Grant 304809/2017-9). R.Var\~ao was partially 
supported 
by National Council for Scientific and Technological Development – CNPq, Brazil 
and partially supported by FAPESP (Grants \#17/06463-3 and \# 16/22475-9).

\bibliographystyle{acm}
\bibliography{referencial}

\begin{thebibliography}{10}

\bibitem{BCEchaotic}
{\sc Buzzi, C.~A., Carvalho, T., and Euz\'{e}bio, R.~D.}
\newblock Chaotic planar piecewise smooth vector fields with non-trivial
  minimal sets.
\newblock {\em Ergodic Theory and Dynamical Systems 36}, 2 (2016), 458--469.

\bibitem{BCEminimal}
{\sc Buzzi, C.~A., Carvalho, T., and Euz\'{e}bio, R.~D.}
\newblock On poincar\'{e}-bendixson theorem and non-trivial minimal sets in
  planar nonsmooth vector fields.
\newblock {\em Publicacions Matem\`{a}tiques 62}, 1 (2018), 113--131.

\bibitem{CarvalhoGoncalvesCristianoTonon-NonLiDy-2020-HIV}
{\sc Carvalho, T., Cristiano, R., Gonçalves, L.~F., and Tonon, D.}
\newblock Global analysis of the dynamics of a mathematical model to
  intermittent hiv treatment.
\newblock {\em Nonlinear Dynamics 101\/} (2020), 719–739.

\bibitem{LF-HairyBall}
{\sc Carvalho, T., and Gonçalves, L.~F.}
\newblock Combing the hairy ball using a vector field without equilibria.
\newblock {\em Journal of Dynamical and Control Systems\/} (2019).

\bibitem{CarvalhoGoncalvesManceraRodrigues-1-2019-cancer}
{\sc Carvalho, T., Gonçalves, L.~F., Mancera, P., and Rodrigues, D.~S.}
\newblock A mathematical model for chemoimmunotherapy of chronic lymphocytic
  leukemia.
\newblock {\em Applied Mathematics and Computation 349\/} (2019), 118--133.

\bibitem{CarvalhoGoncalvesManceraRodrigues-2-2020-cancer}
{\sc Carvalho, T., Gonçalves, L.~F., Mancera, P., and Rodrigues, D.~S.}
\newblock Sliding mode control in a mathematical model to chemoimmunotherapy:
  The occurrence of typical singularities.
\newblock {\em Applied Mathematics and Computation 387\/} (2020), 124782.

\bibitem{CarvalhoGoncalvesNovaes-2020-Shilnikov}
{\sc Carvalho, T., Gonçalves, L.~F., and Novaes, D.~D.}
\newblock Sliding shilnikov connection in filippov-type predator–prey model.
\newblock {\em Nonlinear Dynamics 100\/} (2020), 2973--2987.

\bibitem{Cristiano2019}
{\sc Cristiano, R., Ponce, E., Pagano, D.~J., and Granzotto, M.}
\newblock On the teixeira singularity bifurcation in a dc--dc power electronic
  converter.
\newblock {\em Nonlinear Dynamics 96}, 2 (Apr 2019), 1243--1266.

\bibitem{diBernardo-livro}
{\sc di~Bernardo, M., Budd, C.~J., Champneys, A.~R., and Kowalczyk, P.}
\newblock {\em Piecewise-smooth Dynamical Systems: Theory and Applications},
  first~ed.
\newblock No.~163 in Applied Mathematical Sciences. Springer-Verlag London,
  2008.

\bibitem{EJV}
{\sc Euz\'ebio, R.~D., Juc\'a, J.~S., and Var\~ao, R\'egis, S.}
\newblock There exist transitive piecewise smooth vector fields on s2 but not
  robustly transitive.

\bibitem{EV}
{\sc Euz\'ebio, R.~D., and Var\~ao, R.}
\newblock Topological transitivity imply chaos for two–dimensional filippov
  systems, submitted.

\bibitem{Fi}
{\sc Filippov, A.~F.}
\newblock {\em Differential Equations with Discontinuous Righthand Sides},
  first~ed., vol.~18 of {\em Mathematics and its Applications}.
\newblock Springer Netherlands, 1988.

\bibitem{guardia}
{\sc Guardia, M., Seara, T.~M., and Teixeira, M.~A.}
\newblock \href{https://doi.org/10.1016/j.jde.2010.11.016}{Generic bifurcations
  of low codimension of planar {F}ilippov Systems}.
\newblock {\em Journal of Differential Equations 250}, 4 (2011), 1967--2023.

\bibitem{Katok03}
{\sc Hasselblat, B., and Katok, A.}
\newblock {\em A first course in dynamics: with a panorama of recent
  developments}, first~ed.
\newblock Cambridge University Press, 2003.

\bibitem{kuznetsov}
{\sc Kuznetsov, Y.~A., Rinaldi, S., and Gragnani, A.}
\newblock One-parameter bifurcations in planar filippov systems.
\newblock {\em International Journal of Bifurcation and Chaos 13}, 08 (2003),
  2157--2188.

\bibitem{NV}
{\sc Novaes, D., and Var\~ao, R.}
\newblock A note on invariant measures for filippov systems.
\newblock {\em Bulletin des sciences mathématiques\/} (to appear).

\bibitem{Wang2020}
{\sc Wang, Z., Zhang, C., Zhang, Z., and Bi, Q.}
\newblock Bursting oscillations with boundary homoclinic bifurcations in a
  filippov-type chua's circuit.
\newblock {\em Pramana 94}, 1 (Jul 2020), 95.

\end{thebibliography}

\end{document}